\documentclass[a4paper,twoside,english]{amsart}
\usepackage[T1]{fontenc}
\usepackage{babel}
\usepackage{verbatim}
\usepackage{amsthm}
\usepackage{tikz}
\usepackage{tikz-cd}
\usepackage{amsmath}
\usepackage{amssymb}
\usepackage[unicode=true,
 bookmarks=true,bookmarksnumbered=true,bookmarksopen=false,
 breaklinks=false,pdfborder={0 0 1},backref=false,colorlinks=false]
 {hyperref}
\hypersetup{
 pdfauthor={V. Mantova and U. Zannier}}

\makeatletter

\pdfpageheight\paperheight 
\pdfpagewidth\paperwidth

\theoremstyle{plain}
\newtheorem{thm}{\protect\theoremname}[section]
 \newcommand\thmsname{\protect\theoremname}
 \newcommand\nm@thmtype{theorem}
 \theoremstyle{plain}

  \theoremstyle{remark}
  \newtheorem{rem}[thm]{\protect\remarkname}
  \theoremstyle{definition}
  \newtheorem*{example*}{\protect\examplename}
  \theoremstyle{definition}
  
  \theoremstyle{definition}
  \newtheorem{definition}[thm]{\protect\definitionname}
  \theoremstyle{plain}
  \newtheorem{lem}[thm]{\protect\lemmaname}
  \theoremstyle{plain}
  \newtheorem{prop}[thm]{\protect\propositionname}
  \theoremstyle{plain}
  \newtheorem{cor}[thm]{\protect\corollaryname}
   \theoremstyle{plain}
  
     \theoremstyle{plain}
  \newtheorem{prob}[thm]{\protect\problemname}

\makeatletter
  \theoremstyle{definition}
  \newtheorem{my@rem}[thm]{Remark}
  \renewenvironment{rem}{\begin{my@rem}}{\end{my@rem}}
\makeatother

\makeatother
	\providecommand{\conjecturename}{Conjecture}
  \providecommand{\examplename}{Example}
  	\providecommand{\definitionname}{Definition}
  \providecommand{\lemmaname}{Lemma}
  \providecommand{\propositionname}{Proposition}
  \providecommand{\remarkname}{Remark}
  \providecommand{\theoremname}{Theorem}
\providecommand{\theoremname}{Theorem}
 \providecommand{\corollaryname}{Corollary}
  \providecommand{\problemname}{Problem}

\usepackage[left=3.5cm,right=3.9cm,top=3cm,bottom=3cm]{geometry}

\usepackage{mathtools}

\DeclareMathOperator{\Spec}{Spec}

\def\Z{{\Bbb Z}}

\def\O{{\mathcal O}}

\def\Gal{\mathop{\rm Gal}\nolimits}

\def\CVD{{\hfill\hfil{\lower 2pt\hbox{\vrule\vbox to 7pt
{\hrule width  5pt\varphifill\hrule}\varphirule}}}\par}

\newcommand{\Cu}{\mathcal{C}}
\def\z{{\zeta}}

\title[The Hilbert Property for integral points of affine smooth cubic surfaces]{The Hilbert Property for integral points of affine smooth cubic surfaces}
\author{Simone Coccia}
\begin{document}
\maketitle

\noindent {\bf Abstract}. In this paper we prove that the set of $S$-integral points of the smooth cubic surfaces in $\mathbb{A}^3$ over a number field $k$ is not thin, for suitable $k$ and $S$.
As a corollary, we obtain results on the complement in $\mathbb{P}^2$ of a smooth cubic curve, improving on Beukers' proof that the $S$-integral points are Zariski dense, for suitable $S$ and $k$. With our method we reprove Zariski density, but our result is more powerful since it is a stronger form of Zariski density. 
We moreover prove that the rational integer points on the Fermat cubic surface $x^3+y^3+z^3=1$ form a non-thin set and we link our methods to previous results of Lehmer, Miller-Woollett and Mordell.

\bigskip

\section{Introduction}

In order to state our results, we recall the basic terminology and theorems about thin sets of rational points. Our main references are \cite{B-G}, \cite{corvaja}, \cite{CZ}, \cite{FJ}, \cite{L}, \cite{SeMW}, \cite{SeTGT}, \cite{V}. 

In this paper we assume that algebraic varieties are integral over $k$ and quasi-projective, where $k$ is a number field. By {\it cover} of algebraic varieties $\pi: Y\to X$ we shall mean a {\it dominant rational map of finite degree} with $Y$ irreducible.

\begin{definition}
	Let $X$ be an irreducible $k$-algebraic variety of dimension $n$ and let $\Omega \subseteq X(k)$. We say that $\Omega$ is a \emph{thin} subset of $X(k)$ if there are finitely many irreducible $k$-covers $\pi_i \colon Y_i \to X$ with $\pi_i$ of degree $\ge 2$, such that $\Omega \subseteq \bigcup \pi_i(Y_i(k))$.  
\end{definition}
One can easily prove that if a set of rational points lies in a proper closed subvariety, then it is thin, so that being non-thin is a stronger version of Zariski-density.

\begin{definition}
	We say that an algebraic variety $X$ over $k$ has the \emph{Hilbert property} (HP) if $X(k)$ is not thin.
\end{definition}
It is easy to see that the Hilbert property is a birational invariant.
The classical Hilbert irreducibility theorem can be formulated in this language as the following:
\begin{thm}
	The projective line $\mathbb{P}^1_k$ over a number field $k$ has the Hilbert property.
\end{thm}
From this result one can deduce that all projective spaces (and so all rational varieties) have the HP.

The main inspiration for this work comes from the recent paper \cite{CZ1} of Corvaja and Zannier, where they unveil a relation between the HP and the topology of algebraic varieties, more specifically the ramification of its covers. They prove that for a variety $X$ to have the HP, it is necessary that its non-trivial covers ramify somewhere, i.e. that $X$ is algebraically simply connected. This is a consequence of the Chevalley-Weil theorem, which to the contrary predicts the lifting of all rational points in presence of an unramified cover: the Hilbert property is then an \lq\lq opposite\rq\rq{} of the Chevalley-Weil theorem. Building on this analogy Corvaja and Zannier asked the following:
\begin{prob}[\cite{CZ1}]\label{potente}
	Does every smooth projective algebraically simply connected variety over a number field $k$ with a Zariski dense set of rational points have the HP over $k$?
\end{prob}
See also \S$3.8$ of \cite{CZ} for more discussions about the problem.
In order to understand the power of an affirmative answer, let us say that in the special case of unirational varieties it would imply a positive solution to the Inverse Galois Problem.
Moreover, Corvaja and Zannier gave the first example of non (uni)-rational variety with the HP:
\begin{thm}[\cite{CZ1}]\label{K3}
	The K3 surface defined in $\mathbb{P}^3$ by
	\begin{equation*}
	x^4+y^4=z^4+w^4
	\end{equation*}
	is not (uni)rational over $\mathbb{C}$ and has the Hilbert property over $\mathbb{Q}$.
\end{thm}
In the recent preprint \cite{D} J. L. Demeio provided further examples of varieties with the HP, such as K3 surfaces or quotients of varieties under the action of a finite group.
Corvaja and Zannier formulated analogous definitions and conjectures for the integer case, which are the object of this paper.
\begin{definition}
	We say that an algebraic variety $X/k$ has the \emph{$S$-integral Hilbert property} (IHP) if $X(\O_S)$ is not thin. 
\end{definition}

Equivalently, $X$ has the IHP if given finitely many $k$-covers $\pi_i \colon Y_i \to X$, each without rational sections defined over $\bar k$, the set $X(\mathcal{O}_S)\setminus \bigcup \pi_i(Y_i(k))$ is Zariski-dense. \footnote{	It is immaterial to require that the points are lifted to integral ones, because there is always a finite map which differs from the cover for a closed set, and by Proposition \ref{integerpre} pre-images through finite maps of $S$-integral points are $S'$-integral for an enlargement $S'$ of $S$.}

\begin{rem}\label{riduzione}
	One typical assumption is that the covers are given by finite maps $\pi \colon Y \to X$ where $Y$ is a normal absolutely irreducible variety. In fact, consider a cover $\pi \colon Y \to X$ where $Y$ is an absolutely irreducible variety. By removing a closed set, we can assume that $\pi$ is a dominant morphism. We can mutually decompose $Y$ and $X$ in affine open sets such that $\pi$ restricts to $\pi \colon V \to U$. The map $\pi$ is still a dominant morphism and so it defines an inclusion $\pi^{*}\colon k[U]\hookrightarrow k[V]\subset k(V)=k(Y)$. Taking the integral closure $A$ of $\pi^{*}(k[U])$ in $k(Y)$ one obtains a normal absolutely irreducible affine variety $\tilde V=\Spec(A)$ equipped with a finite morphism $\tilde \pi \colon \tilde V \to U$ and a birational isomorphism $f \colon V \dashrightarrow \tilde V$ such that $\pi=\tilde \pi \circ f$. Gluing these affine spectra we obtain the desired finite map, which differs from $\pi$ only on a closed set, thus proving that this reduction is possible.
\end{rem}

The original version of Hilbert's Irreducibility theorem then becomes:
\begin{thm}\label{hilbert}
	The affine line $\mathbb{A}^1$ has the IHP over $\mathbb{Z}$.
\end{thm}

The natural analogue of Problem \ref{potente} is then:

\begin{prob}\label{scihp}
	Any smooth quasi-projective variety topologically simply connected and with a Zariski-dense set of $S$-integral points has the IHP.
\end{prob}

In order to ask a similar question for non-simply connected varieties, one is led to the following definition.

\begin{definition}[Weak integral Hilbert property]
	A normal quasi-projective algebraic variety $X/k$ has the \emph{weak integral Hilbert property} (WIHP) if, given finitely many covers $\pi_i \colon Y_i \to X$ each ramified above a non-empty divisor, the set $X(\mathcal{O}_S)\setminus \bigcup \pi_i(Y_i(k))$ is Zariski dense.
\end{definition}

\begin{prob}\label{wihp}
	If a normal quasi-projective algebraic variety has a Zariski dense set of $S$-integral points then it has the WIHP.
\end{prob}
\begin{rem}
	Clearly for simply connected varieties the IHP is equivalent to the WIHP.
\end{rem}
\begin{rem}
	It could be possible that we should allow an enlargement of $k$ and $S$ to have a positive solution to the problems, as in our result on smooth cubic surfaces. However in the special case of the Fermat cubic surface we will prove the IHP over $\mathbb{Z}$.
\end{rem}

Let us analyze these problems in the case of curves, as a basic example. For $\mathbb{A}^1$ it is exactly the statement of Theorem \ref{hilbert}.
An interesting example showing the difference between the two is $\mathbb{G}_m$, which is not topologically simply connected, since $\pi_1(\mathbb{G}_m)=\pi_1(\mathbb{C}^{*})=\mathbb{Z}$. In fact $\mathbb{G}_m$ is not an instance of Problem \ref{scihp}, as one can see applying the Chevalley-Weil theorem to the unramified morphism $x \mapsto x^n$ from $\mathbb{G}_m$ to $\mathbb{G}_m$. However it is a positive example for Problem \ref{wihp}, since a ramified cover of $\mathbb{G}_m$ has always finitely many $S$-integral points. This is a consequence of Siegel's theorem.
For any other curve Siegel's theorem provides a complete solution, since the $S$-integral points are finite and so the IHP cannot hold.

We finally state the first result of this paper, which is an instance of Problem \ref{scihp} (in its potential version).
\begin{thm}\label{fregno}
	Let $\widetilde F \subset \mathbb{P}^3$ be a smooth cubic surface defined over a number field $k$ and let $H$ be a smooth hyperplane section of $\widetilde F$ (which is a planar cubic curve contained in $\widetilde F$). We define $F\coloneqq \widetilde F \setminus H$. Then there exists an extension $K$ of $k$ and a finite set of places $S$ of $K$ such that $F(\mathcal{O}_S)$ is Zariski-dense and $F$ has the IHP for the $S$-integral points.
\end{thm}
It will be proven in the next section, following the method of Corvaja and Zannier of the two elliptic fibrations, which in our case are substituted with conic fibrations.
This is a potential result since it could be possible that $F(\O_S)$ is Zariski-dense for smaller $K$ and $S$, but we could have to take a finite extension of the field and enlarge $S$ to ensure that it is non-thin. However the Fermat cubic surface, i.e. the surface in $\mathbb{A}^3$ defined by
\begin{equation}\label{eqferm}
	x^3+y^3+z^3=1,
\end{equation}
will provide a non-trivial instance of a surface satisfying the statement of Problem \ref{scihp} in its full form. This is the content of the following

\begin{thm}\label{fermat}
	The rational integer points on the Fermat cubic surface form a non thin set.
\end{thm}  

The study of the diophantine equation \eqref{eqferm} dates back to Euler, who proved the rationality of the Fermat cubic surface (see section $13.7$ of \cite{hardy}). The interest for this equation was renewed in the fifties by Mordell (\cite{mordell2}) who analyzed more general cubic equations, such as
\begin{equation*}
	x^3+y^3+z^3=k
\end{equation*}
At Mordell's suggestion, Millet and Woollett (\cite{miwo}) found all solutions for $0 \le k \le 100$ and $\lvert z \rvert \le \lvert y \rvert \le \lvert x\rvert \le 3164$ with the help of the electronic computer EDSAC at the Cambridge University Mathematical Laboratory. In the case of the Fermat cubic surface they found 21 non-trivial solutions, i.e. satisfying $(x+y)(y+z)(z+x)\neq 0$, of which $8$ are given by the parametric solution
\begin{equation}\label{parametric}
	x=9t^4, \qquad y=-9t^4+3t, \qquad z=-9t^3+1
\end{equation}
where $t$ is an integer.
It followed another paper of Mordell (\cite{mordell1}) where he proved the infinitude of integer solutions of
\begin{equation*}
	ax^3+ay^3+bz^3=bc^3
\end{equation*}
and observed that four of the solution of Miller-Woollett satisfied
\begin{equation*}
	z-1=-7(x+y)
\end{equation*}
suggesting that these could be found by the method of his paper.
Finally in \cite{lehmer} Lehmer exhibited an infinite sequence of parametric solutions of which \eqref{parametric} is the first. All these solutions satisfy the relation
\begin{equation*}
	1-z=3t^2(x+y)
\end{equation*}
where $t$ is an integer. He pushed forward the observation of Mordell, noticing that the table contains $9$ solutions of the form
\begin{equation}\label{linear}
	\alpha(x+y)=1-z
\end{equation}
where $\alpha$ is an integer.
With our method we are able to find three families of rational curves with infinitely many integral points. As we will see later, equation \eqref{linear} corresponds to one of these families, so that our method subsumes the one of Lehmer and Mordell. It is interesting to note that their method to produce parametric solutions is based on the theory of Pell's equations, as ours, tough implicitly. In fact Theorem \ref{interi} is just a generalization to rational curves of the infinitude of integral solutions of the Pell's equations.

In the last section we are going to prove the WIHP for the complement in $\mathbb{P}^2$ of a smooth cubic, which will be a consequence of Theorem \ref{fregno}. The Zariski density of the $S$-integral points was already proved by Beukers in \cite{beukers} by constructing a family of conics with infinitely many $S$-integral points. However they form a thin set for a cover ramified outside the cubic, so that our result is stronger.

\bigskip 

\noindent{ \bf Acknowledgements}. My deepest thanks go to my advisor Prof. Pietro Corvaja, for having suggested me this topic as my \textit{Tesi di laurea} and having guided me with endless support through the literature and the pitfalls in which I ended up. I would also like to thank Julian Demeio and Prof. Francesco Zucconi for helpful discussions and for the interest shown in this work.

\section{Some results on integral points}

Here we briefly recall the basic definitions and results on integral points on varieties. Our main reference is \cite{corvaja}.

Let $v$ be a finite place of a number field $k$ and $\mathcal{O}_v$ be the valuation ring at $v$, i.e. $\mathcal{O}_v=\{x\in k \colon \lvert x \rvert_v \le 1\}$. We denote by $\mathfrak{m}_v=\{x\in k \colon \lvert x \rvert_v<1 \}$ its maximal ideal and by $k(v)=\mathcal{O}_v/\mathfrak{m}_v$ its residue field (which is finite).
Let us then choose a finite set $S$ of places containing the archimedean ones and define as
\begin{equation*}
\mathcal{O}_S=\{x\in k \colon \lvert x\rvert_v \le 1, \forall v \notin S \}
\end{equation*}
the ring of $S$-integers.
Let $D \subset \mathbb{P}^n$ be a closed subvariety over $k$ which is defined by a homogeneous ideal $I_D \subset k[X_0,\dots,X_n]$. We can consider $I_{D,v}\coloneqq I_D \cap \mathcal{O}_v[X_0,\dots,X_n]$ and its image in the quotient ring $\mathcal{O}_v[X_0,\dots,X_n]/\mathfrak{m}_v=k(v)[X_0,\dots,X_n]$. It is still a homogeneous ideal and so it defines a closed subvariety $D_v$ of $\mathbb{P}^n$ over $k(v)$. We say that $P$ reduces modulo $v$ to $D$ if $P_v \in D_v$, where $P_v$ is the reduction modulo $v$ of $P$.

Let $\widetilde X$ be a projective $k$-variety in $\mathbb{P}^n$ and $D \subset \widetilde X$ a subvariety. We say that $P \in \widetilde X$ is integral with respect to $D$ if for no finite place $v$ the point $P$ reduces to $D$. If $S$ is a finite set of places containing the archimedean ones, we say that $P$ is $S$-integral with respect to $D$ if it never reduces to $D$ for any prime \emph{outside} $S$. Usually one writes $X=\widetilde X \setminus D$ and by $X(\mathcal{O}_S)$ we mean the set of $S$-integral points of $\widetilde X$ with respect to $D$. If $D$ is a divisor, it is called \emph{divisor at infinity}, in the sense that one completes the quasi-projective variety $X$ to $\widetilde X$ adding the \lq\lq points at infinity\rq\rq{} contained in $D$.
When $X \subset \mathbb{A}^n \subset \mathbb{P}^n$ is an affine variety, we recover the naive definition using coordinates. Indeed, considering the completion $\widetilde X$ under the canonical embedding $\mathbb{A}^n \hookrightarrow \mathbb{P}^n$ and defining $D=\widetilde X \setminus X$, the $S$-integral points of $\widetilde X$ with respect to $D$ will be the points of $X$ with coordinates in $\mathcal{O}_S$.

We now study the behavior of integral points under morphisms, allowing if necessary an enlargement of $S$.
Let $X_1$ and $X_2$ be quasi-projective algebraic varieties defined over a number field $k$ and let $\pi \colon X_1 \to X_2$ be a $k$-morphism. We can complete both varieties to projective varieties $\widetilde X_1$ and $\widetilde X_2$ in such a way that $\pi$ can be continued to a morphism $\pi \colon \widetilde X_1 \to \widetilde X_2$ with $\pi^{-1}(D_2)=D_1$, where $D_i\coloneqq \widetilde X_i \setminus X_i$.
Usually we will omit any reference to $\widetilde X$ and $D$: any time we will have a morphism between (quasi-projective) varieties it will be implicit that it has been continued to the completions and the divisors at infinity.

We now list some results that we will use in our proofs.

\begin{prop}
	Let $X_1$ and $X_2$ be quasi-projective algebraic varieties defined over a number field $k$ and let $\pi \colon X_1 \to X_2$ be a $k$-morphism. Let $S$ be a finite set of places containing the archimedean ones. Enlarging $S$ if necessary, we have $\pi (X_1(\mathcal{O}_S))\subseteq X_2(\mathcal{O}_S)$. 
\end{prop}

\begin{cor}\label{belo}
	Let $X_1$ and $X_2$ be quasi-projective algebraic varieties defined over a number field $k$ and let $\pi \colon X_1 \to X_2$ be a $k$-morphism. Let $S$ be a finite set of places containing the archimedean ones.
	If $\pi$ is dominant and $X_1(\mathcal{O}_S)$ is Zariski-dense, then $X_2(\O_S)$ is Zariski-dense, possibly after an enlargement of $S$. Conversely, if $\pi$ is dominant and $X_2(\O_S)$ is never Zariski-dense for any $S$, then $X_1(\O_S)$ cannot be Zariski-dense.
\end{cor}

\begin{prop}\label{integerpre}
	Let $X_1$ and $X_2$ be quasi-projective algebraic varieties defined over a number field $k$.
	Let $\pi \colon X_1 \to X_2$ be a finite map and $S$ a finite set of places containing the archimedean ones. Then there is another finite set of places $S'$, containing $S$, such that $\pi^{-1}(X_2(\O_{S'}))\cap X_1(k) \subseteq X_1(\O_{S'})$. In other terms, the rational preimages of the $S'$-integral points of $X_2$ are $S'$-integral points of $X_1$.
\end{prop}

\begin{prop}\label{trick}
	Let $\Cu$ be a geometrically irreducible rational curve over $k$ with (at least) two points in the divisor at infinity. Let $\pi \colon \mathcal{D} \to \mathcal{C}$ be a finite $k$-morphism of geometrically irreducible curves ramified above (at least) one point in the affine part of $\mathcal{C}$. Then $\mathcal{D}$ has only finitely many $S$-integral points.
\end{prop}
\begin{proof}
	We apply Siegel's theorem: if $\mathcal{D}$ has genus $>0$ we are done, otherwise we prove that $\mathcal{D}$ has at least three points at infinity. Using Hurwitz formula, we have $2(d-1)=\sum_x (e_x-1)>0$, where $e_x$ is the ramification index at $x$ and $d=\deg \pi$, so that $d>1$. Moreover the numerical relation implies that we cannot have total ramification in the two points at infinity, since there is a ramification point in the affine part of $\mathcal{C}$. Then $\mathcal{D}$ has at least three points at infinity.
\end{proof}

\section{Proof of Theorem \ref{fregno}}

First, we need to prove the (potential) Zariski density of $S$-integral points. This result is not new, since it is implicit in Beukers \cite{libro} and explicit in \cite{Hats} of Hassett-Tschinkel, where they proved the following generalization for Del Pezzo surfaces:

\begin{thm}[\cite{Hats}]
	Let $X$ be a smooth Del Pezzo surface and $D$ a smooth anticanonical divisor. Then integral points for $(X, D)$ are potentially dense.
\end{thm}

Here we present our own construction of a Zariski dense set of $S$-integral points, since it is essential in our proof of Theorems \ref{fregno} and \ref{fermat}.

\subsection{The density of $S$-integral points}

We recall that a smooth cubic surface in $\mathbb{P}^3$ can be realized, over a finite extension of $k$, as the blowup of $\mathbb{P}^2$ in six points $P_1, \dots, P_6$ in general position. By general position we mean that there are no three collinear points and there is no conic passing through all of them. We then have a morphism $\pi \colon \widetilde F \to \mathbb{P}^2$ which is one to one everywhere except over the six points, where the exceptional divisors $E_i=\pi^{-1}(P_i)$ are lines contained in the cubic surface.
The map $\pi$ is birational and its inverse
\begin{equation*}
\pi^{-1} \colon \mathbb{P}^2 \dashrightarrow \widetilde F
\end{equation*}
is given by $[r:s:t] \longmapsto [f_{0}(r,s,t): \dots : f_{3}(r,s,t)]$, where $f_0, \dots, f_3$ are four linearly independent cubic homogeneous polynomials vanishing at the $6$ points. As a consequence, the set of indeterminacy of $\pi^{-1}$ is given by the six points.
The image of $H$ through $\pi$ is a cubic curve $D$ in $\mathbb{P}^2$ passing through the six points. 

We enlarge $k$ so that $P_1,P_2,P_5,P_6$ are defined over $k$ and the couple $\left\{P_3,P_4\right\}$ is defined over $k$. Then we have that the blowup is defined over $k$ 
and the linear systems of conics through $P_1,P_2,P_3,P_4$ and $P_3,P_4,P_5,P_6$ are defined over $k$.

There are $15$ lines in $\mathbb{P}^2$ connecting pairs of points in $\left\{P_1,\dots, P_6\right\}$. The strict transforms of these lines correspond to some of the $27$ lines in $\widetilde F$ and intersect the divisor at infinity $H$ in one point. In fact, if the line is not tangent in any of the two blown-up points, it will intersect $D$ in one more point, which corresponds to a point in $H$. Otherwise, if it is tangent in one of the two points, there will be no more intersection with $D$, and the only point at infinity of the strict transform will be the one above the tangency point.

Let $\ell$ be the line joining $P_1$ and $P_6$ (which is then defined over $k$). Its strict transform $L$ in $F$ is then isomorphic to $\mathbb{A}^1$, i.e. to $\mathbb{P}^1$ with one point at infinity. Enlarging $S$ we can make $L(\mathcal{O}_S)$ infinite and by Hilbert irreducibility theorem $L(\mathcal{O}_S)$ is a non-thin set.

We then consider the morphism $\lambda \colon \widetilde F \to \mathbb{P}^1$ given by the composition of the map $\pi$ and the rational map $\mathbb{P}^2 \dashrightarrow \mathbb{P}^1$ which sends the point $P \in F$ to the intersection of $\ell$ and the conic of the linear system through $P$ (we shall soon see that $\lambda$ is indeed a morphism). More precisely we map a point $Q \in F$ to $\pi(Q) \in \mathbb{P}^2$ and consider the conic through $P_1,P_2,P_3,P_4$ and $\pi(Q)$. The intersection between $\ell$ and the conic is given by two points, of which one is $P_1$, which is defined over $k$. Then, if $Q$ is defined over $k$, by the hypotheses on $k$, also the second intersection point will be defined over $k$, and we define it to be the image of the map $\lambda$. The map $\mathbb{P}^2 \dashrightarrow \mathbb{P}^1$ is not a morphism, since it is not defined in $P_1,P_2,P_3,P_4$, but $\lambda$ is a morphism. To show this, we only have to check that it can be extended to the exceptional divisors over $P_1,P_2,P_3,P_4$. In fact, for a given point of, say, $E_1$, there passes exactly one strict transform of a conic of the linear system, and we consider this conic for the above construction.
Thus $\lambda$ defines a morphism $\widetilde F \to \mathbb{P}^1$ with fibers given by the strict transforms of the conics of the linear system. We will write $\Cu_l$ for the fiber corresponding to the value $l$.

We can define in an analogous way the fibration $\mu \colon \widetilde F \to \mathbb{P}^1$ whose fibers are given by the strict transform of the conics through $P_3,P_4,P_5,P_6$. We denote these curves with $\Cu'_m$.

We now observe that $\Cu_l$ is a curve with two points at infinity, since the corresponding conic in $\mathbb{P}^2$ intersects $D \setminus \{P_1,\dots,P_6\}$ in two points. There are only finitely many exceptions that can occur, namely the conic through $P_5$, the one through $P_6$ and the conics tangent to the cubic. Then we have that the general $\Cu_l$ is isomorphic to $\mathbb{G}_m$, i.e. to $\mathbb{P}^1$ minus two points at infinity. It is natural then to study which $\Cu_l$ have infinitely many $S$-integral points. 
In order to do this, we recall the following well known result.

\begin{thm}\label{interi}
	Let $\Cu$ be a smooth rational curve with two points at infinity $P,P'$.
	Suppose that either
	\begin{enumerate}
		\item $\O_S^{*}$ is infinite,
		\item or the points at infinity $P$ and $P'$ are quadratic conjugates and one place of $S$ splits in the corresponding quadratic extension.
	\end{enumerate}
	Then if $\Cu$ contains an $S$-integral point, it has infinitely many ones.
\end{thm}
\begin{proof}
	A smooth rational curve with two points at infinity admits a planar model as a conic $\Cu$ with two points at infinity. We denote by $L$ the line passing through these two points. Then we can use Corollary $5.3.3$ of Chapter $5$ of \cite{corvaja}.
\end{proof}

The hypotheses of this theorem are easily verified for every $\Cu_l$ if we enlarge $S$ so to have $\O_S^{*}$ infinite. Then we will have that any curve $\Cu_l$ passing through an $S$-integral point will contain infinitely many other $S$-integral points. 
The set of curves $\Cu_l$ such that $\Cu_l(\O_S)$ is infinite is parametrized by a non-thin (and thus infinite) set of values $l \in k$, since a curve $\Cu_l$ through $p \in L(\O_S)$ has infinitely many $S$-integral points, and $L(\O_S)$ is not thin over $k$. This gives the Zariski-density of $F(\O_S)$.

Clearly the same properties hold for the curves $\Cu'_m$.

\begin{rem}
	The reader may ask why we did not use the pencil of lines through a blown-up point. The reason will be clear in the proof, where we will use the simple connectedness of $F$ minus some lines. Still, one can prove the theorem using a system of conics and a system of lines. However, in the case of the Fermat cubic surface, there is no system of lines defined over $\mathbb{Q}$, so we preferred this proof.
\end{rem}

\subsection{Topology of complements of curves in smooth cubic surfaces}
 
Let us denote by $L_{3,4}$ the strict transform of the line through $P_3$ and $P_4$. We write
\begin{equation*}
	X=F\setminus L_{3,4}=\widetilde{F}\setminus \left(H \cup L_{3,4} \right)
\end{equation*}
We want to prove the following:
\begin{thm}\label{top}
	The surface $X$ is topologically simply connected.
\end{thm}

We recall the following two lemmas.

\begin{lem}\label{quoziente}
	Let $X$ be a (connected) complex manifold, $Z \subset X$ a proper closed complex submanifold, $Y = X\setminus Z$ its complement, $p \in Y$ a point. Then the inclusion $i \colon Y \hookrightarrow X$ induces a surjective homomorphism $i_{*}\colon \pi_1(Y,p) \rightarrow \pi_1(X,p)$ between the corresponding fundamental groups.
\end{lem}
\begin{proof}
	See Lemma $2$, section $3$ of \cite{CZ2}.
\end{proof} 
\begin{lem}\label{ausiliario}
	The surface $U \coloneqq X \setminus \left(E_1 \cup \cdots \cup E_6 \right)$ has fundamental group isomorphic to $\mathbb{Z}$.
\end{lem}
\begin{proof}
	The surface $U$ is isomorphic to the complement of a smooth cubic and a transverse line in $\mathbb{P}^2$, so that the result is a consequence of the remark following Theorem $2$ of \cite{SeAb}.
\end{proof}

\begin{lem}\label{mail}
	The surface $X$ admits no finite cyclic unramified connected cover of degree $>1$.
\end{lem}
\begin{proof}
	Let $Y \to X$ be a cyclic unramified cover of degree $d>1$. By a Theorem of Grauert and Remmert (see \cite{SeTGT}, Ch.$6$), $Y$ has a structure of algebraic variety such that the covering map is algebraic. Then the function field of $Y$ is obtained by taking the $d$-th root of a rational function $f \in \mathbb{C}(\widetilde{X})^{*}=\mathbb{C}(\widetilde{F})^{*}$ on $\widetilde{F}$. The property that $Y \to X$ is unramified means that all zeros and poles of $f$ in $X$ have multiplicity divisible by $d$. This means that for each irreducible curve $C \subset \widetilde F$ outside $L_{3,4} \cup H$ the multiplicity of $C$ in the divisor $(f)$ is divisible by $d$. Hence we have
	\begin{equation*}
		(f)=aL_{3,4}+bH+dD
	\end{equation*}
	for a divisor $D$. For an exceptional divisor $E_i$ with $i \neq 3,4$ we have the following intersection products
	\begin{equation*}
		L_{3,4}^2= -1, \, L_{3,4}.H=1, \, L_{3,4}.E_i=0, \,  E_i.H=1
	\end{equation*}
	Multiplying $(f)$ by $L_{3,4}$ and $E_i$ we obtain 
	\begin{align*}
		&0=(f).L_{3,4}=-a+b+dD.L_{3,4}\\
		&0=(f).E_i=b+dD.E_i\\
	\end{align*}
	so that $d$ must divide both $a$ and $b$. Then there would not be ramification over $L_{3,4}$ and $H$ and $Y \to X$ would define an unramified cover of all $\widetilde F$, which is simply connected, so that the cover would be trivial.
\end{proof}
\begin{proof}[Proof of Theorem \ref{top}]
	Let $U \subset X$ be the surface of Lemma \ref{ausiliario}. By Lemma \ref{quoziente} the fundamental group of $X$ is a quotient of the fundamental group of $U$, which is isomorphic to $\mathbb{Z}$ after Lemma \ref{ausiliario}, so that $\pi_1(X)$ is a quotient of $\mathbb{Z}$. To prove that such a quotient is trivial, it is sufficient to show that it admits no non-trivial finite cyclic quotients. Geometrically, this means that $X$ admits no non-trivial (finite) cyclic cover, which is the content of Lemma \ref{mail}.
\end{proof}
\begin{cor}
	The surface $F$ is simply connected.
\end{cor}
We then have proved that $F$ satisfies the hypotheses of Problem \ref{scihp}, so that we can proceed to the proof of the IHP.

\subsection{The IHP for cubic surfaces}

Our proof of Theorem \ref{fregno} follows the method developed by Corvaja and Zannier in \cite{CZ1} to prove the Hilbert property for the Fermat quartic surface while now the role of the elliptic fibrations is played by fibrations in conics with two points at infinity. One interesting difference is that here the topology of complements of curves in $F$ is crucial for the success of the strategy. In fact $F$ is simply connected, so that any cover of degree $>1$ must ramify outside $H$. In other terms, for $(F,H)$ the IHP is equivalent to the WIHP. Moreover Theorem \ref{top} will be used in an essential step of the following proof.

We argue by contradiction. Suppose we are given finitely many finite maps $\pi_i \colon Y_i \to F$ of degree $>1$, where $Y_i$ is an absolutely irreducible variety, such that $F(\O_S) \setminus \bigcup_i(\pi_i(Y_i(k)))$ is not Zariski-dense (and so it is contained in a curve). Then, as we observed, the maps $\pi_i$ will be ramified over a non-empty divisor of $F$. We can suppose that the projective continuation $\widetilde \pi_i \colon \widetilde Y_i \to \widetilde F$ are finite\footnote{In fact the usual reduction of Remark \ref{riduzione} produces a projective variety because $\widetilde F$ is projective.}, too.
In this case, when we speak about $S$-integral points on $Y_i$, we mean that $\widetilde \pi_i^{-1}(H)$ is the divisor at infinity of $\widetilde Y_i$. Enlarging the primes $S\subset S'$ we may suppose that if the pre-image through $\pi_i$ of a point in $F(\O_S)$ lies in $Y_i(k)$, then it actually belongs to $Y_i(\O_{S'})$.

Let us consider one such cover $\pi \colon Y \to F$. The generic fiber of $\lambda \circ \pi \colon Y \to \mathbb{P}^1$ can be either reducible or irreducible, i.e. the fiber $Y(l)\coloneqq \pi^{-1}(\Cu_l)$ is reducible or irreducible for general $l \in \mathbb{C}$. We divide the covers in three types:
\begin{enumerate}
	\item $Y(l)$ is generically reducible.
	\item $Y(l)$ is generically irreducible and there exists an irreducible curve in the ramification divisor of $\pi$ which generically intersects the curves $\Cu_l$.
	\item $Y(l)$ is generically irreducible and there exists no such curve as above. More explicitly, the components of the ramification divisor can be irreducible curves of the family $\Cu_l$, strict transforms of the lines through the four base points of the linear system or the exceptional divisors $E_5$ and $E_6$. To prove this, let us first observe that $E_5$ and $E_6$ are the only exceptional divisors which do not generically intersects the curves $\Cu_l$. Let $\mathcal{D}$ be an irreducible curve in $F$, distinct from the exceptional divisors, which does not generically intersect the curves $\Cu_l$. Then $\pi(\mathcal{D})$ is a curve in $\mathbb{P}^2$ which does not generically intersect the conics $\pi(\Cu_l)$. However these are the conics of a linear system, so they generically intersect every curve which is not part of their linear system. Then $\pi(\mathcal{D})$ can be an irreducible conic $\pi(\Cu_l)$ or a reducible one, i.e. a line through the four base points. Passing to strict transforms we have the assertion. 
\end{enumerate}

First, we treat the case of generically reducible covers.
\begin{lem}
	If the generic fiber of $\lambda \circ \pi$ is reducible, there exists a cover defined over $k$ of $k$-irreducible curves $\phi \colon \Cu \to \mathbb{P}^1$ such that the pullback of $Y$ over the base change $\Cu \times_{\mathbb{P}^1}\widetilde F$ is reducible over $\bar k$.
\end{lem}
\begin{proof}
	Irreducibility is a birational invariant, hence we reduce to the case of a projection of affine varieties $\pi \colon Y \to X$ where $X \coloneqq (f(x,y,\lambda)=0) \subseteq \mathbb{A}^3$ and $Y \coloneqq (f(x,y,\lambda)=0)\cap(g(x,y,z,\lambda)=0) \subseteq \mathbb{A}^4$, $\pi$ maps $(x,y,z,\lambda)$ to $(x,y,\lambda)$ and $\lambda \colon X \to \mathbb{P}^1$ maps $(x,y,\lambda)$ to $\lambda$. The fact that $\lambda \circ \pi$ has reducible generic fiber is equivalent to the fact that for all but finitely many $\lambda_0 \in \bar k$ the fiber $(f(x,y,\lambda_0)=0)\cap(g(x,y,z,\lambda_0)=0)$ is reducible. Up to a birational isomorphism we can eliminate one variable and reduce to a surface defined by $P(u,v,t)=0$, projecting $(u,v,t)$ to $t$, in such a way that for all but finitely many $t_0 \in \bar{k}$ the polynomial $P(u,v,t_0)$ is reducible. This implies that there exists a non-trivial factorization
	\begin{equation*}
	P(u,v,t)=A_t(u,v) \cdot B_t(u,v)
	\end{equation*} 
	where $A_t$ and $B_t$ are polynomials in $u$ and $v$ with coefficients given by algebraic functions of $t$. Since there are finitely many such coefficients, each providing a finite algebraic extension of the scalars $k(t)$, there is an algebraic element $s$ over $k(t)$, such that the above factorization is valid in $k(t,s)$. Geometrically, this corresponds to a $k$-cover $\Cu \to \mathbb{P}^1$, where $\Cu$ is a $k$-irreducible curve with function field $k(t,s)$.
	\begin{equation*}
	\begin{tikzcd}
	Y' \arrow{r} \arrow{d}
	&Y \arrow{d}{\pi}\\
	\Cu \times_{\mathbb{P}^1} X \arrow{r} \arrow{d} &X \arrow{d}{\lambda}\\
	\Cu \arrow{r}{\phi} &\mathbb{P}^1
	\end{tikzcd}
	\end{equation*}
	From the above discussion, the pullback $Y'$ of $Y$ with respect to this base change is birationally equivalent to $(P(u,v,t)=0)$, hence reducible over $\bar{k}$.
\end{proof}
Using this Lemma and the fact that $Y$ is geometrically irreducible, we have $\deg(\phi \colon \Cu \to \mathbb{P}^1)>1$. Then the components of $Y(l)$ may be defined over $k$ only for $l \in \phi(\Cu(k))$, which is a thin set in $\mathbb{P}^1(k)$. That is, for $l$ outside this thin set, the fiber $Y(l)$ is irreducible over $k$ but reducible over $\bar{k}$, and this implies that $Y(l)$ has only finitely many rational points.\footnote{See the first remark in Chapter $3$ of \cite{SeTGT}.}

We denote by $T$ the union of these thin sets, which in the following we will enlarge with finitely many other thin sets, if necessary.
We then have obtained that for $l \notin T$ all but finitely many points of $\Cu_l(k)$, and in particular of $\Cu_l(\O_S)$, are not lifted to the generically reducible covers. Then choosing $l \notin T$ with $\Cu_l(\O_S)$ infinite (this always exists since $L(\O_S)$ is not thin), we have that infinitely many integral points do not lift.

Let us examine the covers with generically irreducible fibers. The set of $l \in \mathbb{C}$ for which $Y(l)$ is reducible is a closed set, then it is finite and we can put it in $T$.

We begin with covers of the second type. Let $B$ be one irreducible curve in the ramification divisor of $\pi$, which is not contained in $H$ and generically intersects the curves $\Cu_l$. We claim that $Y(l)$ has finitely many $S'$-integral points, except for finitely many $l$ which we put in $T$. In order to do this, we have to prove that for all but finitely many $l$, the restriction of $\pi$ to $Y(l)$ defines a cover of $\Cu_l$ ramified over $B \cap \Cu_l$. We use the following lemma.

\begin{lem}\label{restriction}
	Let $\pi \colon Y \to X$ be a cover of surfaces which is a finite map. Let $B$ be one irreducible component of the ramification divisor. Let $\Cu$ be a curve in $X$ intersecting $B$ transversally at some point $p$ which is smooth for both $\Cu$ and $B$. Suppose that $\pi^{-1}(\Cu)$ is absolutely irreducible. Then the restriction of $\pi$ to $\pi^{-1}(\Cu)$ defines a cover $\pi \colon \pi^{-1}(\Cu) \to \Cu$ ramified above $p$.
\end{lem}
\begin{proof}
	There is a neighborhood $U$ of $p$ and local analytic coordinates $(x,y)$ for which $p=(0,0)$, $\Cu$ is defined by $x=0$, $B$ is defined by $y=0$ and $U \setminus (U\cap B)$ is biholomorphic to $D\times D^*$ where $D$ is the open unit disc in $\mathbb{C}$ and $D^*=D \setminus \left\{0\right\}$. For a ramification point $q \in \pi^{-1}(p)$ there will be local analytic coordinates $(u,v)$ such that $\pi^{-1}(\Cu)$ is defined by $u=0$ and the map $\pi$ is given locally by $x=u$ and $y=v^e$ where $e>1$ is the ramification index at $q$. Then the pre-image of $\Cu$ is a smooth curve at $q$ and $\pi \colon \pi^{-1}(\Cu) \to \Cu$ ramifies at $q$.
\end{proof}

Clearly the generic $\Cu_l$ intersects transversally the ramification divisor $B_i$ (of the second type cover $\pi_i$) in smooth points not contained in the divisor at infinity $H$. We put the finitely many exceptions in $T$.
Then for $l \notin T$ we can apply Lemma \ref{restriction} to obtain that the restriction $\pi_i \colon \pi_i^{-1}(\Cu_l) \to \Cu_l$ is a cover of curves ramified above an affine point of $\Cu_l$. By Proposition \ref{trick} we have that $\pi_i^{-1}(\Cu_l)$ has only finitely many $S'$-integral points, for all $i$'s. Recall that we enlarged $S$ so that the $k$-rational preimages of $S$-integral points are $S'$-integral. Then only finitely many points of $\Cu_l(\O_S)$ can lift to the covers of second type.
Then, if all the covers were of second type, we would have a contradiction considering the curves $\Cu_l$ with $l \notin T$ and infinitely many $S$-integral points.

Let us then study the covers of third type.
First we observe that the points of total ramification, i.e. with exactly one pre-image through $\widetilde \pi$, form a closed subset of the ramification locus. Then $\widetilde \pi$ totally ramifies on all of $H$ or on finitely many points of $H$. In this last case, if we put in $T$ the $l$'s for which $\Cu_l$ passes through the total ramification points, we have that for $l \notin T$ the absolutely irreducible curve $\pi^{-1}(\Cu_l)$ has more than two points at infinity, and so it has only finitely many $S'$-integral points.

We can then assume that third type covers are totally ramified over $H$. 
For all $l$, except the finitely many ones for which $\Cu_l$ is tangent to $H$ or it is a component of the ramification locus, $Y(l)$ is an absolutely irreducible curve with two points at infinity and is unramified in the affine part over $\Cu_l$. Then $Y(l)$ is isomorphic to a multiplicative torus $\mathbb{G}_m$ which we denote by $\mathcal{T}_l$ and $\pi_{\vert \mathcal{T}_l}$ is the translate of an isogeny. By the assumption on the surjectivity of the covers on integral points, the $S$-integral points which do not lift are contained in a curve, so that, for $l \in k\setminus T$, all but finitely many $S$-integral points of $\Cu_l$ lift to rational points of the $\mathcal{T}_l$'s (of the various covers). Moreover, if $\Cu_l(\O_S)$ is infinite then it is isomorphic to a finitely generated abelian group of positive rank and $\pi(\mathcal{T}_l(O_S))$ is the translate of a subgroup of $\Cu_l(\O_S)$.

Then we have that when $\Cu_l(\O_S)$ is infinite, \emph{all} the points in $\Cu_l(\O_S)$ are lifted, as it follows from the following (Lemma $3.2$ of \cite{CZ1}).

\begin{lem}[\cite{CZ1}]
	Let $G$ be a finitely generated abelian group of positive rank and, for $u$ in a finite set $U$, let $H_u$ be subgroups of $G$ and let $h_u\in G$. Suppose that  $G \setminus\bigcup_{u\in U}(h_u+H_u)$ is finite. Then this complement is actually empty.
\end{lem}

We recall that with $\Cu'_m$ we denoted the curves associated to the other fibration $\mu$.
By construction $\Cu'_m$ intersects the possible ramification divisors (different from $H$) of third type covers, with the only exception of the line $L_{3,4}$ through $P_3$ and $P_4$, in a finite set, which is nonempty for all but finitely many $m$. However any cover of $\widetilde{F}$ must ramify on a curve outside $L_{3,4}\cup H$, since $\widetilde F\setminus \left(L_{3,4}\cup H\right)$ is simply connected, as proved in Theorem \ref{top}.
Then each cover of the third type will be ramified in the affine part of $\Cu'_m$ for all but finitely many $m$. Repeating the argument we used for the covers of second type, we have that for $m$ outside a thin set $T'$, only finitely many points in $\Cu'_m(\O_S)$ lift to rational points of covers of the third type. At the same time any $p \in \Cu'_m(\O_S)$ lies also on $\Cu_l$, where $l=\lambda(p)$, and since it is an $S$-integral point, if $l \notin T$ then $\Cu_l(\O_S)$ is infinite.
Therefore, if there exists $m \in k\setminus T'$ with $\Cu_m'(\O_S)$ infinite, such that $\lambda(p)\notin T$ for infinitely many $p \in \Cu'_m(\O_S)$, we would have a contradiction, since all the points in $\Cu_{\lambda(p)}(\O_S)$ are lifted.

We can then suppose that, for any given  $m\in k\setminus T'$ with $\Cu_m'(\O_S)$ infinite, all but finitely many points  $p\in \Cu'_m(\O_S)$ are such that $\lambda(p)\in T$.
We have that $T$, being a thin set, is contained in a finite union of sets $\phi_j(Z_j(k))$ where $Z_j$ is an absolutely irreducible curve and $\phi_j \colon Z_j \to L$ is a finite map of degree $>1$, ramified over a finite subset of $L \cong \mathbb{P}^1$. We write $\mathcal{R}$ for the union of those branch points, which is a finite set.

The set of branch points of $\lambda$ restricted to $\Cu'_m$ depends {\it a priori} on $m$. Indeed, it turns out that each branch point depends on $m$. This is the analogue of Lemma $3.3$ of \cite{CZ1}.

\begin{lem}\label{geom} For all $m$ such that $\Cu'_m$ is smooth, the restriction to $\Cu'_m$ of $\lambda$ is a  map of degree two, having two ramification points. The value  of $\lambda$ at each of these ramification points is a non-constant (algebraic) function of $m$.
\end{lem}

\begin{proof} We provide a geometric proof of this statement. Let us fix an $m$ so that $\Cu'_m$ is smooth. In order to compute the degree of the map $\lambda$ restricted to $\Cu'_m$, we have to compute the number of points in $\Cu'_m$ where it takes a generic value. Now, by our opening construction, a `value' of the map $\lambda$ is represented by a point of $L$, i.e. by the curve $\Cu_l$ passing through this point. Now $\pi(\Cu_l)$ and $\pi(\Cu_m)$ are two conics passing through $P_3$ and $P_4$, so they generically intersect in two more points. Taking the strict transform one gets two points of intersection and so the map has degree two. The ramified values of $\lambda$ are the conics $\pi(\Cu_l)$ which are tangent to $\pi(\Cu_m')$ in points different from $P_3$ and $P_4$, so that we have exactly two ramification points.
	
We want to prove that these  values are not constant as $m$ varies.
Now, if a ramified value for $\lambda$ were fixed, say equal to $l$, then $\pi(\Cu_l)$ would be tangent to $\pi(\Cu'_m)$ for all values of $m \in \mathbb{P}^1$. But this is impossible since reversing the reasoning with $\mu$ in place of $\lambda$, we have that there are exactly two values of $m$ for which this happens, corresponding to the ramified values of $\mu$.
\end{proof}

Analogously, the $\lambda$-images of the two points at infinity of $\Cu_m'$ are given by non-constant functions of $m$.
As a consequence we have that the set of branch points of $\lambda$ restricted to $\Cu_m'$ and the set of values of $\lambda$ at the points at infinity can intersect $\mathcal{R}$ only for finitely many $m$. Let $m \notin T'$ be one value for which this does not happen and $\Cu'_m(\O_S)$ is infinite. Let $\phi \colon Z \to \mathbb{P}^1$ be one of the morphisms above. The fiber product $W \to \mathbb{P}^1$ of $\lambda\colon \Cu'_m \to \mathbb{P}^1$ and $\phi\colon Z\to \mathbb{P}^1$ has the property that the image of $W(k)$ is $\phi(Z(k))\cap \lambda(\Cu'_m(k))$.
The map $\phi$, having degree $>1$, must ramify over at least two points which are not branch points for $\lambda$, nor images of the two points at infinity, which we denote by $Q$ and $Q'$. Hence the restriction to any irreducible component $V \subset W$ of the map $\pi \colon W \to \Cu'_m$ must ramify somewhere in the affine part of $\Cu'_m$. Let $D=\pi^{-1}(Q+Q')$ be the divisor given by the pullback of the divisor at infinity of $\Cu'_m$. Then there is a set of places $S''$ containing $S$ such that the integral points of $\Cu'_m$ which lifts, actually lift to $V(\O_{S''})$. In symbols, this means $\pi^{-1}(\Cu'_m(\O_S))\cap V(k)\subset V(\O_{S''})$. However using Proposition \ref{trick} we have that $V(\O_{S''})$ is finite, so, repeating the argument for all irreducible components of $W$, we conclude that only finitely many points of $\Cu'_m(\O_S)$ lift to rational points of $W$.

This proves that $\lambda(\Cu'_m(\O_S))\cap \phi(Z(k))$ is finite for each morphism $\phi$, hence $\lambda(\Cu'_m(\O_S))\cap T$ is finite. However this contradicts the assumption that all but finitely many points $p$ in $\Cu'_m(\O_S)$ (which is an infinite set) satisfy $\lambda(p)\in T$.

In all cases we were able to derive a contradiction, then proving the theorem.

\section{Proof of Theorem \ref{fermat}}

Theorem \ref{fregno} is valid for all smooth cubic surfaces, but it is \emph{potential}, since we had to enlarge the base field and $S$. Of course, it is still possible that $F(\O_S)$ is not thin for smaller $k$ and $S$. This is the case for the Fermat cubic surface, which is the variety defined over the rational numbers $\mathbb{Q}$ by the equation
\begin{equation*}
x^3+y^3+z^3+w^3=0
\end{equation*}
In the literature it is also referred to as the surface $x^3+y^3+z^3=w^3$, but this is clearly isomorphic to the one above under the involution $w \mapsto -w$, so that it is immaterial to study one or the other. We take $S=\{\infty\}$ and $H \coloneqq (w=0)$ as the divisor at infinity, so that the $S$-integral points are exactly the points with integer coordinates on the affine surface
\begin{equation*}
x^3+y^3+z^3+1=0
\end{equation*}  
First, we provide the explicit formulas of the blowup, which can be found at the link \cite{Elk} of Noam Elkies' webpage.
We denote by $r,s,t$ the homogeneous coordinates in $\mathbb{P}^2$. Then the map $\mathbb{P}^2 \dashrightarrow F$ is given by 
\begin{align*}
w&=-(s+r)t^2 + (s^2+2r^2) t - s^3 + rs^2 - 2r^2s - r^3 \\
x&=t^3 - (s+r)t^2 + (s^2+2r^2) t + rs^2 - 2r^2s + r^3\\
y&=-t^3 + (s+r)t^2 - (s^2+2r^2) t + 2rs^2 - r^2s + 2r^3\\
z&=(s-2r)t^2 + (r^2-s^2) t + s^3 - rs^2 + 2r^2s - 2r^3
\end{align*}
and its birational inverse $F \to \mathbb{P}^2$ by
\begin{align*}
r&=yz - wx\\
s&=wy - wx + xz + w^2 - wz + z^2\\
t&=y^2 - xy + wy + x^2 - wx + xz
\end{align*}
unless $w+y$, $x+z$ are both zero, in which case $r,s,t$ are proportional to $x+y,y,x$.

This map $[w:x:y:z] \mapsto [r:s:t]$ blows down the six disjoint lines, as follows: 
\begin{itemize}
	\item	$E_1 \colon w+\z z = x+\z y = 0$ \, to \, $P_1=[-\z:1:1]$; 
	\item	$E_2 \colon w+\bar\z z = x+\bar\z y = 0$ \, to \, $P_2=[-\bar\z:1:1]$; 
	\item	$E_3 \colon x+\z z = y+\z w = 0$ \, to \, $P_3=[0:1:-\z]$; 
	\item	$E_4 \colon x+\bar\z z = y+\bar\z w = 0$ \, to \, $P_4=[0:1:-\bar\z]$; 
	\item	$E_5 \colon y+\z z = w+\z x = 0$ \, to \, $P_5=[1:-\bar\z:-\z]$; 
	\item	$E_6 \colon y+\bar\z z = w+\bar\z x = 0$ \, to \, $P_6=[1:-\z:-\bar\z]$; 
\end{itemize}
where $\z$ and $\bar \z$ are the primitive cubic roots of unity.

Observe that the couples of points $(P_1,P_2)$, $(P_3,P_4)$ and $(P_5,P_6)$ are paired under the action of $\Gal(\mathbb{Q}(\z)/\mathbb{Q})$, so that the only lines connecting two blown-up points which are defined over $\mathbb{Q}$ are the three lines connecting these pairs. These are parametrized as
\begin{itemize}
	\item $[r:s:s]$, which maps to $[-x:x:y:-y]$ with $x=-(r+s)$, $y=2r-s$,
	\item $[0:s:t]$, which maps to $[s:-t:t:-s]$,  
	\item $[s+t:s:t]$ which maps to $[s:-t:-s:t]$.
\end{itemize} 
Similarly, of the fifteen pencils of conics passing through four of the six blown-up points, only three are defined over $\mathbb{Q}$ and they are the ones passing through two of the above couples.
To produce infinitely many integer points we want to apply Theorem \ref{interi} to the curves of the blow-up of these three pencils defined over $\mathbb{Q}$, but we have to check which curves satisfy the hypotheses. In fact, the group of units $\mathbb{Z}^{*}=\left\{1,-1 \right\}$ is finite, so we have to use the condition $(2)$ in Theorem \ref{interi}. Since we are searching for the rational integer points, $S$ consists in the archimedean place of $\mathbb{Q}$, so that we have to hope that the archimedean place splits in the extension at infinity, or equivalently that the extension is real quadratic. This means that the discriminant of the quadratic polynomial defining the extension at infinity is non-square and positive. Of course, since the blow-up is an isomorphism outside the blown-up points, the points at infinity are quadratic conjugates if and only if their images in $\mathbb{P}^2$ are so. Then we will check the positivity of the discriminant using the equations of the curves in $\mathbb{P}^2$. 

First, observe that the points on the line $[w:x:y:z]=[s:-t:-s:t]$ which are integral with respect to $w=0$ are exactly those for which $[1:-t/s:-1:t/s]$ has integer coordinates, i.e. $[s:t]=[1:n]$ where $n$ is an integer. The corresponding points in $\mathbb{P}^2$ have coordinates $[n+1:1:n]$ with $n \in \mathbb{Z}$.

The idea is to select a conic of the pencil trough $P_1,P_2,P_3,P_4$ which passes trough an integral point of the above line. If condition on the discriminant is satisfied, then, since the blown-up conic has one integer point, we can apply Theorem \ref{interi} to show that it has infinitely many. Repeating the argument for all the integer points of the line we obtain the sought Zariski dense set of integer points. It remains to check condition $(2)$ of Theorem \ref{interi}. 

The linear system of conics through $P_1,P_2,P_3,P_4$ has equation
\begin{equation}\label{coni}
C_{[a:b]}\colon a(-rs + rt) + b(r^2 - rs + s^2 - st + t^2)=0
\end{equation}
so that the conic $C_n$ through $[n+1:1:n]$ is given by
\begin{equation*}
[a:b]=[2n^2+1:1-n^2].
\end{equation*}

After some calculations we find that the discriminant of the quadratic equation for the points at infinity is
\begin{equation}\label{delta}
\Delta=\frac{-36u^3-54u+9}{(2u+1)^3}
\end{equation}
where $u=b/a$. The solution for $\Delta>0$ is
\begin{equation*}
-\frac{1}{2}< u < 0.1637399876771411...
\end{equation*}
In particular for
\begin{equation}\label{un}
u=\frac{1-n^2}{2n^2+1}
\end{equation}
the condition is satisfied for all $n \neq 0$. We have to check when $\Delta$ is non-square. Substituting \eqref{un} in \eqref{delta} we find
\begin{equation*}
	\Delta=12n^6-3
\end{equation*}
The diophantine equation
\begin{equation*}
	m^2=12n^6-3
\end{equation*}
has only finitely many integer solutions, since it represents an hyperelliptic curve of genus $2$. Then for all sufficiently large integer $n$ we have that the strict transform $\Cu_n$ of the conic $C_n$ passing through the integral point $[n+1:1:n]$ satisfies the hypotheses of Theorem \ref{interi} and so it has infinitely many integer points. Via this geometric construction we have proved
\begin{thm}
	The integer points $F(\mathbb{Z})$ are Zariski dense in the Fermat cubic surface.
\end{thm}
This result is not new, since it is implicit in the paper \cite{lehmer} of Lehmer. He proved that the curves in $F$ of equation\footnote{Here we have changed the sign of the variables with respect to the formulas in the introduction, since these referred to $x^3+y^3+z^3=1$, while we are studying $x^3+y^3+z^3+1=0$.}
\begin{equation*}
	\begin{cases}
	1+z=-3m^2(x+y)\\
	x^3+y^3+z^3+1=0
	\end{cases}
\end{equation*}
with $m \in \mathbb{Z}$, have infinitely many integer points, from which it follows the Zariski density. This method is not different from ours, since Lehmer's curves are strict transforms of conics passing through $P_1,P_2,P_5,P_6$.

Unfortunately the integer points of the union of the curves $\Cu_n$ for $n=1,2,\dots$ are not sufficient for the IHP over $\mathbb{Z}$, since they form a thin set. In fact, we can consider the surface
\begin{equation*}
Y=\{(p,q)\in F \times L_{5,6} \, \vert \, q \in \Cu_{\lambda(p)}\cap L_{5,6}  \}
\end{equation*} 
where $L_{5,6}$ is the line passing through $P_5$ and $P_6$. Then $Y$ is endowed with the natural projection $\rho \colon Y \to F$ over the first coordinate. This is a degree two morphism, ramified over the curve $\Cu_l$ tangent to $L_{5,6}$ and it clearly lifts all the integral points we produced.

However this method is much more powerful. In fact we have produced an infinite number of curves $\Cu_n$ with infinitely many integral points, and we still dispose of two other families of conics, the ones through $P_1,P_2,P_5,P_6$ and the ones through $P_3,P_4,P_5,P_6$. Then we can repeat the above construction for both families using \emph{every} curve $\Cu_n$ in place of the line $[s:-t:-s:t]$. Of course we have to study when the hypotheses of Theorem \ref{interi} are satisfied for the other two pencils. We then provide a detailed analysis of these pencils of conics.

Let us begin with the conics through $P_1,P_2,P_5,P_6$. They are expressed by the equations
\begin{equation*}
D_{[a:b]}\colon a(s-t)(r-s-t) + b(t^2-tr+r^2)
\end{equation*}
where $[a:b]\in \mathbb{P}^1$. The line through the points at infinity is
\begin{equation*}
ar+(b+2a)s-(b+a)t=0
\end{equation*}
and the discriminant at infinity is given by
\begin{equation*}
\Delta=-3u^4 - 12u^3 - 18u^2 + 9
\end{equation*}
where $u=b/a$. Solving the inequality $\Delta>0$ we find
\begin{equation*}
-1<u < 0.587401055408971...
\end{equation*}
In particular, it is sufficient that $u$ lies in
\begin{equation*}
-1< u< \frac{1}{2}
\end{equation*}
From the definition of $u$ we have that
\begin{equation*}
u=\frac{(t-s)(r-s-t)}{t^2-tr+r^2}
\end{equation*}
for any $[r:s:t] \in D_{[a:b]}$, so that, if there exists an integer point in $\Cu_n(\mathbb{Z})$ such that its blown-down point $[r:s:t]$ satisfies 
\begin{equation*}
-1< \frac{(t-s)(r-s-t)}{t^2-tr+r^2}< \frac{1}{2}
\end{equation*}
then the strict transform $\mathcal{D}_{[a:b]}$ of the conic $D_{[a:b]}$ through this point will have positive discriminant at infinity and so, if $\Delta$ is non-square, by Theorem \ref{interi} it will have infinitely many integral points. Let us then study the above inequalities.

First, let us observe that $t^2-tr+t^2$ is always positive, so that the inequality is equivalent to
\begin{equation*}
-t^2+tr-r^2<tr-t^2-sr+s^2< \frac{1}{2}(t^2-tr+r^2)
\end{equation*}
The left-hand one is
\begin{equation*}
r^2-rs+s^2>0
\end{equation*}
which is always true. The right-hand one is
\begin{equation*}
3t^2-3tr+r^2+2sr-2s^2>0
\end{equation*}
which in the affine plane $(s \neq 0)$ becomes
\begin{equation*}
3t^2-3tr+r^2+2r-2>0
\end{equation*}
This defines the outer part of an ellipse, while from \ref{coni} the equation defining $C_n$ in the affine plane $s \neq 0$ is 
\begin{equation*}
n^2(-r^2 + 2rt - r - t^2 + t - 1) + r^2 + rt -2r + t^2 - t + 1=0
\end{equation*}
and for $n$ sufficiently large it is an hyperbola lying in the area outside the ellipse. Thus we have proved that for almost all $n$, the curves of the family $\mathcal{D}$ passing through an integral point of $\Cu_n$ have positive discriminant at infinity.

We now repeat the same argument for the conics through $P_3,P_4,P_5,P_6$. They are expressed by the equations
\begin{equation*}
E_{[a:b]}\colon ar(t+s-r)+b(4r^2-2rt-2rs+t^2-ts+s^2)
\end{equation*}
where $[a:b]\in \mathbb{P}^1$. The line through the points at infinity is
\begin{equation*}
as + b(r - s - t)=0
\end{equation*}
and the discriminant at infinity is given by
\begin{equation*}
\Delta=u^4 - 18u^2 + 36u - 27
\end{equation*}
where $u=a/b$. Then we can solve the inequality $\Delta>0$ to find
\begin{equation*}
u \in \left]-\infty,-5.107243650047037...\right[ \, \cup \,  \left]3,+\infty\right[
\end{equation*}
In particular it is sufficient that $u$ lies in
\begin{equation*}
u \in \left]-\infty,-6\right[ \, \cup \,  \left]3,+\infty\right[
\end{equation*}
From the definition of $u$ we have that
\begin{equation*}
u=\frac{4r^2-2rt-2rs+t^2-ts+s^2}{r(r-s-t)}
\end{equation*}
for any $[r:s:t] \in E_{[a:b]}$. As above we have reduced to study inequalities. Let us first consider the second one
\begin{equation*}
\frac{4r^2-2rt-2rs+t^2-ts+s^2}{r(r-s-t)}>3
\end{equation*}
We work in the affine plane $(s\neq 0)$, i.e. we normalize with respect to $s$. After some manipulations we get
\begin{equation*}
\frac{r^2+rt+t^2+r-t+1}{r(r-1-t)}>0
\end{equation*}
The numerator is always positive since
\begin{equation*}
r^2+rt+t^2+r-t+1=\left( \frac{r}{\sqrt{2}}+\frac{t}{\sqrt{2}}\right)^2+\left( \frac{r}{\sqrt{2}}+\frac{1}{\sqrt{2}}\right)^2+\left( \frac{t}{\sqrt{2}}-\frac{1}{\sqrt{2}}\right)^2>0
\end{equation*}
so that the above inequality is really
\begin{equation*}
r(r-1-t)>0
\end{equation*}
The other inequality is
\begin{equation*}
\frac{4r^2-2rt-2r+t^2-t+1}{r(r-1-t)}<-6
\end{equation*}
Now suppose $r(r-1-t)<0$ so that multiplying we get
\begin{equation*}
4r^2-2rt-2r+t^2-t+1>-6r(r-1-t)
\end{equation*}
which is
\begin{equation*}
10r^2-8rt-8r+t^2-t+1>0
\end{equation*}
One easily see
that for almost all $n$, $C_n$ lies in the union of
\begin{equation*}
\big(r(r-1-t)>0\big) \cup \big(\big(r(r-1-t)<0\big) \cap \big(10r^2-8rt-8r+t^2-t+1>0\big)\big)
\end{equation*}
so that the condition on the discriminant is satisfied.

Before finishing the proof, we are going to show the relation between our method and Lehmer's construction of families of integer solutions. He starts with the solutions
\begin{equation}\label{para}
x=-9m^4, \qquad y=9m^4-3m, \qquad z=9m^3-1
\end{equation}
where $m$ is an integer. This parametrizes the curve
\begin{equation*}
	\begin{cases}
	(x+y)^4+9x=0\\
	x^3+y^3+z^3+1=0
	\end{cases}
\end{equation*}
giving an isomorphism with $\mathbb{A}^1$ with inverse
\begin{equation*}
	(x,y,z) \mapsto -\frac{x+y}{3}
\end{equation*}
The curve has only one point at infinity, namely $[0:1:-1:0]$. If we blow down the curve, we find that its image is the conic
\begin{equation*}
	-2r^2+r(s+t)+st=0
\end{equation*}
It passes through the two blown-up points $P_1$ and $P_2$ and, having only $[0:1:-1:0]$ at infinity, it has one more intersection point with the cubic in $[0:0:1]$, with tangency multiplicity $4$.
For a given $m$, Lehmer produces an infinite number of parametric solutions on the curve
\begin{equation*}
1+z=-3m^2(x+y)
\end{equation*}
starting from the solution \eqref{para} and using the theory of Pell's equations.
If we blow-down this curve, we find
\begin{equation*}
	3m^2(r^2-rs+s^2)-(r^2-rt+t^2)=0
\end{equation*}
which is exactly the conic $D_{[a:b]}$ with $[a:b]=[-3m^2:3m^2-1]$. Then
\begin{equation*}
	u=\frac{b}{a}=\frac{3m^2-1}{-3m^2}=\frac{1}{3m^2}-1
\end{equation*}
satisfies the condition on the positivity of the discriminant
\begin{equation*}
	-1<\frac{1}{3m^2}-1<\frac{1}{2}
\end{equation*}
so that our method applies to these curves.

We now come to the proof of Theorem \ref{fermat}. We consider the morphism $\lambda \colon F \to \mathbb{P}^1$ given by the composition of the blow-down map and the rational map to $\mathbb{P}^1$ associated with the linear system. Geometrically we map a point $Q \in F$ to the projective point $[a:b]$ defining the conic $D_{[a:b]}$ through $\pi(Q)$ and $P_1,P_2,P_5,P_6$. \emph{A priori} $\lambda$ is a rational map, since it could not be defined over the exceptional divisors $E_1,E_2,E_5,E_6$. However for a given point of, say, $E_1$, there passes exactly one strict transform of a conic of the linear system, so we can map this point to the projective point associated with that conic.
Thus $\lambda$ can be extended to a morphism $\lambda \colon F \to \mathbb{P}^1$ with fiber $\mathcal{D}_{[a:b]}$ over $[a:b]$.
Analogously, we call $\mu \colon F \to \mathbb{P}^1$ the map with fibers $\mathcal{E}_{[a:b]}$.

We now prove that $\lambda(\{p \in F(\mathbb{Z})\vert \exists n \in \mathbb{N} \,\, p \in \Cu_n \})=\lambda(\bigcup_n \Cu_n(\mathbb{Z}))$ is not thin in $\mathbb{P}^1$. In other words, the set of points in $\mathbb{P}^1$ parametrizing the curves of $\mathcal{D}$ passing through an integer point of $\Cu_n$ for some $n$ is not thin. Moreover the curves of $\mathcal{D}$ whose discriminant is a square are parametrized by a thin set, since they are given by the image of the projection from the curve
\begin{equation*}
	-3u^4 - 12u^3 - 18u^2 + 9=m^2
\end{equation*}
to $\mathbb{P}^1$ mapping $(u,m)$ to $[1:u]=[a:b]$.
The same of course is true for $\mu$, and so we have produced two independent fibrations with a non thin set of fibers (i.e. parametrized by a non-thin set) with infinitely many integer points. We can then apply the same proof 
of Theorem \ref{fregno} to deduce Theorem \ref{fermat}.

Suppose by contradiction that $T=\lambda(\bigcup_{n}\Cu_n(\mathbb{Z}))$ is thin. Then it is contained in a finite union of sets $\phi_i(Z_i(\mathbb{Q}))$, where $Z_i$ is an absolutely irreducible curve and $\phi_i \colon Z_i \to \mathbb{P}^1$ is a finite map of degree $>1$, ramified over a finite subset of $\mathbb{P}^1$. We write $\mathcal{R}$ for the union of those branch points.
The analogue of Lemma \ref{geom} is:
\begin{lem}
	Let $\Cu_t\coloneqq \Cu_{[1:t]} $ be the pencil defined by the equation \eqref{coni}. For $\Cu_t$ smooth, the restriction of $\lambda$ to $\Cu_t$ is a degree two map with two ramification points. The value of $\lambda$ at any of these points is a non-constant algebraic function of $t$.
\end{lem}
Analogously, the $\lambda$-images of the two points at infinity of $\Cu_n$ are given by non-constant functions of $n$.
As a consequence we have that the set of branch points of $\lambda$ restricted to $\Cu_n$ and the set of values of $\lambda$ at the points at infinity can intersect $\mathcal{R}$ only for finitely many $n$. Let $n$ be one integer for which this does not happen. Let $\phi \colon Z \to \mathbb{P}^1$ be one of the morphisms above. The fiber product $W \to \mathbb{P}^1$ of $\lambda\colon \Cu_n \to \mathbb{P}^1$ and $\phi\colon Z\to \mathbb{P}^1$ has the property that the image of $W(\mathbb{Q})$ is $\phi(Z(\mathbb{Q}))\cap \lambda(\Cu_n(\mathbb{Q}))$.
The map $\phi$, having degree $>1$, must ramify over at least two points which are not branch points for $\lambda$, nor images of the two points at infinity. Hence the restriction to any irreducible component $V \subset W$ of the map $\pi \colon W \to \Cu_n$ must ramify somewhere and not above the points at infinity $Q$ and $Q'$. Let $D=\pi^{-1}(Q+Q')$ be the divisor of $V$ given by the pullback of the divisor at infinity of $\Cu_n$. Then there is a set of places $S$, containing the place at infinity, such that the integral points of $\Cu_n$ which lift to rational points of $V$, actually lift to $V(\Z_S)$. In symbols, this means $\pi^{-1}(\Cu_n(\mathbb{Z}))\cap V(\mathbb{Q})\subset V(\Z_S)$. However using Proposition \ref{trick} we have that $V(\Z_S)$ is finite so, repeating the argument for all irreducible component of $W$, we conclude that only finitely many points of $\Cu_n(\Z)$ lift to rational points of $W$.

This proves that $\lambda(\Cu_n(\Z))\cap \phi(Z(k))$ is finite for each morphism $\phi$, hence $\lambda(\Cu_n(\Z))\cap T$ is finite. However this is absurd since $\lambda(\Cu_n(\Z))\subset T$ and $\Cu_n(\Z)$ is infinite.

As stated above, this concludes the proof of the theorem, since we have obtained two independent fibrations with a non thin set of fibers with infinitely many integer points, and the same proof of Theorem \ref{fregno} applies.

\begin{rem}
	Our method for the construction of the non-thin set of $S$-integral points has a geometrical interpretation. Indeed, one chooses
	\begin{enumerate}
		\item a point $P$ on the line,
		\item a point $Q$ on the curve of $\Cu$ passing trough $P$,
		\item a point $R$ on the curve of $\mathcal{D}$ passing trough $Q$.
	\end{enumerate}
	Then there are three parameters all varying in $\mathbb{P}^1$, since the lines and conics can be identified with $\mathbb{P}_1$. Geometrically, we get a rational map $\mathbb{P}_1^3 \dashrightarrow F$ where $\mathbb{P}_1^3 \coloneqq \mathbb{P}_1\times \mathbb{P}_1 \times \mathbb{P}_1$, which sends the three parameters to the point on the curve of $\mathcal{D}$ identified by the above construction. We then have two maps $\mathbb{P}_1^3 \dashrightarrow F$, one for $\mathcal{D}$ and one for $\mathcal{E}$.
	This delineates a difference with the method of Theorem \ref{fregno}, where we considered the integral points on a line and then the conics trough these points, so that we had two rational maps $\mathbb{P}_1^2 \dashrightarrow F$.
	This is also the case for the proof of Theorem \ref{K3} of Corvaja and Zannier, where they use two parameters for the construction of the non-thin set of rational points.
\end{rem}

\section{The WIHP for the complement of the smooth plane cubic}

In \cite{beukers}, Beukers proved the following.
\begin{thm}\label{beu}
	Let $D \subset \mathbb{P}^2$ be a curve of degree $\le 3$ with at most normal crossing singularities, defined over a number field $k$, and let $X=\mathbb{P}^2\setminus D$. Then there exists a finite extension of $k$ and a finite set of places $S$ such that $X(\mathcal{O}_S)$ is Zariski-dense in $X$.
\end{thm}

Clearly, it suffices to study the complements of maximal degree $\deg D=3$, since Zariski density is easier to accomplish with a smaller $D$. The four cases are:
\begin{enumerate}
	\item three lines in general position,
	\item a smooth conic and a non tangent line,
	\item a singular irreducible cubic,
	\item a smooth cubic.
\end{enumerate}

In all these cases, the strategy of the proof relies on constructing an infinite family of rational curves with two points at infinity (i.e. intersecting $D$, the divisor at infinity) and with infinitely many $S$-integral points. This gives the Zariski-density of $X(\mathcal{O}_S)$.

Once obtained the Zariski-density, one could try to see if these surfaces provide a positive instance of Problem \ref{wihp}.
\begin{thm}
	The complements of the curves above have the WIHP.
\end{thm}
The first three cases reduce to the WIHP for $\mathbb{G}^2_m$. This is already established and there are much more general results regarding algebraic groups (see \cite{Col,C,F-Z,Z}).
The most difficult case is the one of the smooth cubic.
In fact one cannot prove the WIHP for this surface using the conics with infinitely many $S$-integral points found by Beukers. The obstruction is that they form a thin set covered by a map with ramification divisor not contained in the smooth cubic $D$. In other words, the WIHP for $X$ cannot be proved using only the $S$-integral points coming from Beukers' conics. With our method we provide a different proof of the potential Zariski density and we are able to prove the WIHP as a corollary of Theorem \ref{fregno}.

In fact, suppose that $D$ is defined in $\mathbb{P}^2$ by the equation $D=(G(x,y,z)=0)$. Then the smooth cubic surface $\widetilde F \subset \mathbb{P}^3$ of equation $S=(w^3=G(x,y,z))$ is naturally endowed with the projection morphism $\pi \colon \widetilde F \to \mathbb{P}^2$ sending $[x:y:z:w]$ to $[x:y:z]$. The pre-image of $D$ is given by $H =(w=0) \cap F$, hence it is a smooth hyperlane section of $\widetilde F$.
This map is (totally) ramified only above $D$, so $F=\widetilde F \setminus H$ defines a cyclic \'{e}tale covering of degree $3$ of $X$, and since $F$ is simply connected, it is the universal cover of $X$. We can apply Theorem \ref{fregno} to show that for an extension $k \subset K$ and a set of places $S$, the surface $F$ has the IHP. Since $\pi$ is dominant, we can apply Corollary \ref{belo} to show that there exists an enlargement of $S$ (which we still denote by $S$) such that $\pi(F(\O_{S}))\subset X(\O_{S})$. It follows $X(\O_{S})$ is Zariski-dense.

The WIHP is a consequence of the IHP for $F$.
In fact, suppose we are given finitely many finite maps $\pi_i \colon Y_i \to X$ lifting all the $S$-integral points of $X$, ramified outside $D$.
We consider the fiber product $F \times_{X} Y_i$, which fits in the following diagram
\begin{equation*}
\begin{tikzcd}
F \times_{X} Y_i \arrow{r}{p_{i,2}} \arrow{d}{p_{i,1}}
&Y_i \arrow{d}{\pi_i}\\
F \arrow{r}{\pi} &X
\end{tikzcd}
\end{equation*}
We assert that all $S$-integral points of $F$ lift to the covers $p_{i,1}$. Indeed, let us consider an $S$-integral point $q \in F(\O_S)$. By hypotheses the image $\pi(q)$ is $S$-integral in $X$. Then there is a cover $\pi_i$ and a rational point $y \in Y_i(k)$ such that $\pi_i(y)=\pi(q)$. Therefore the point $(q,y)$ lies in $F \times_X Y_i$ and $p_{i,1}(q,y)=q$. Moreover the covers $p_{i,1}$ are not trivial, since $\pi^{-1}(D)=H$ and $\pi_i$ is ramified outside $D$, and this contradicts the IHP for $F$.

\vfill

Simone Coccia

Scuola Superiore dell'Universit\'{a} di Udine

Palazzo Garzolini di Toppo Wassermann

Via Gemona 92

33100 Udine - ITALY

coccia.simone@spes.uniud.it

\end{document}